\documentclass[11pt]{amsart}
\usepackage[a4paper,margin=25mm]{geometry}
\usepackage{amsmath,amssymb,amsthm}
\usepackage[T1]{fontenc}
\usepackage{lmodern,microtype,cite,hyperref}
\newtheorem{theorem}{Theorem}
\newtheorem{proposition}[theorem]{Proposition}
\newtheorem{lemma}[theorem]{Lemma}
\theoremstyle{remark}
\newtheorem{remark}[theorem]{Remark}
\newcommand{\F}{\mathbb F}
\newcommand{\Z}{\mathbb Z}
\newcommand{\C}{\mathbb C}
\newcommand{\one}{\mathbf 1}
\newcommand{\vp}{v_p}
\DeclareMathOperator{\Ann}{Ann}
\DeclareMathOperator{\GR}{GR}
\title{Subset sums of power values over Galois rings}
\author[Zijie Gu]{Zijie Gu}

\makeatletter
\g@addto@macro\@setauthors{%
\begin{center}
\small
Independent Researcher, Shanghai China\\{}
\href{mailto:hilightgu@gmail.com}{hilightgu@gmail.com}\\{}
\end{center}
}
\makeatother
\date{}
\begin{document}
\begin{abstract}
Let $R=\GR(p^m,p^{mf})$, where $p$ is prime, $q=p^f$, $m,f\geq1$, $d\geq2$ and $p\nmid d$. We count $n$-element subsets of $P=\{x^d:x\in R\}\setminus\{0\}$, the set of distinct nonzero $d$th-power values, with a prescribed sum. For fixed $m,d,n$ with $n\geq3$, the count equals
\[
      q^{-m}\binom{|P|}{n}(1+O_{m,d,n}(q^{1-n/2}))
      \]
as $q\to\infty$ through admissible prime powers, uniformly over all targets and with both $p$ and $f$ allowed to vary. For fixed $m,d$, we also classify uniform relative equidistribution at the remaining positive sizes: it fails at $n=1$, and at $n=2$ it holds along an admissible sequence if and only if, for all sufficiently large terms, $\gcd(d,q-1)=1$ and $(m\geq2\text{ or }p\ne2)$.
\end{abstract}
\maketitle

\section{Introduction}\label{sec:introduction}

Let $R=\GR(p^m,p^{mf})$, where $p$ is prime, $q=p^f$, $m,f\geq1$, $d\geq2$ and $p\nmid d$. We use $R/pR\cong\F_q$ and call $m$ the length of $R$. Put $P=\{x^d:x\in R\}\setminus\{0\}$ and $N=|P|$. For $b\in R$ and an integer $n\geq0$, put
\[
 N_P(n,b)=\#\left\{S\subseteq P:|S|=n,\quad\sum_{x\in S}x=b\right\}.
\]
Distinctness is imposed on the power values themselves, including any nonzero nonunit power values.

Write $g=\gcd(d,q-1)$ for the index of $(\F_q^\times)^d$ in $\F_q^\times$. A prime power $q=p^f$ is admissible when $p\nmid d$. For fixed $m,d$ and positive $n$, uniform relative equidistribution along an admissible sequence $q\to\infty$ means that
\[
  \frac{N_P(n,b)}{q^{-m}\binom Nn} \to 1
\]
uniformly over $b\in R$. Relative normalisation is restricted to $1\leq n\leq N$, which holds eventually for fixed positive $n$.

The finite-field antecedents include Li and Wan's exact count over nonzero field elements \cite[Theorem~1.2]{LW08} and Zhu and Wan's subgroup estimates \cite[Theorems~4.1 and~4.4]{ZW12}, whose cycle weights give the corresponding fixed-size field rates, including characteristic two. Exact quadratic-residue counts in odd characteristic appear in Weiqiong Wang, Li-Ping Wang and Haiyan Zhou \cite[Theorems~5.1 and~5.2]{WWZ17}. Weiqiong Wang and Jennifer Nguyen \cite[Theorem~3.1]{WN18} and Hyok Choe and Chunghyok Choe \cite[Theorem~3.1]{CC19} give small-size all-target subgroup positivity criteria in odd and binary characteristic, respectively, under explicit field-size and index hypotheses.

At length two in odd characteristic, Ding and Zhou's results \cite[Theorem~6.1]{DZ} and \cite[Theorem~5.11 and Corollary~5.12]{DZII} cover the same distinct nonzero power-value set, with main term $q^{-2}\binom Nn$. Their combined exact formulas give the same coefficient majorant and all-target fixed-size upper rates $O_{d,n}(q^{1-n/2})$ for $g>1$ and $O_n(q^{1-n})$ for $g=1$, including $p\mid n$. In particular, the bound for \(g=1\) implies uniform relative equidistribution when \(n=2\). These are extracted consequences of exact formulas, distinct from printed compressed estimates or separately stated rate theorems. Appendix~\ref{app:length-two-comparison} gives the comparison and the qualified target scope of the retained 2019 derivation.

The counting background comprises the weighted distinct-coordinate sieve and cycle-index formulas \cite[Theorem~2.6, Proposition~2.8 and Lemma~4.1]{LW10}, and the character-product method \cite[Lemma~2.5]{Kosters}. Our signed character-and-cycle identity specialises Li and Wan's general finite-ring formula \cite[Lemma~3.1]{LW18}.

\medskip\noindent\textbf{Main results.}
For fixed $m,d,n$ with $n\geq3$, Theorem~\ref{thm:rates} gives
\[
 N_P(n,b)=q^{-m}\binom Nn
 \left(1+O_{m,d,n}(q^{1-n/2})\right),
\]
uniformly over every $b\in R$, as $q\to\infty$ along any admissible sequence. The prime $p$, extension degree $f$ and pointwise index $g$ may all vary. For fixed $m,d$ and a fixed positive size $n$, the complete classification of uniform relative equidistribution over every target along any admissible sequence $q_j=p_j^{f_j}\to\infty$ is given by Theorem~\ref{thm:classification}:
\begin{itemize}
\item At $n=1$, uniform relative equidistribution always fails.
\item At every $n\geq3$, it always holds.
\item At $n=2$, it holds if and only if, eventually, both
\[
 g_j=\gcd(d,q_j-1)=1,\qquad
 (m\geq2\text{ or }p_j\ne2).
\]
The successful pair cases have relative error $O_{m,d}(q_j^{-1})$.
\end{itemize}

Theorem~\ref{thm:rates} also provides explicit finite bounds, sharper estimates when $g=1$, and sufficient conditions ensuring positivity for every target. Proposition~\ref{prop:exponent} determines the order of an explicit nonnegative majorant for the relative error, while Proposition~\ref{prop:obstructions} establishes the necessity of the conditions in the classification through obstructions arising from the actual subset-sum counts.

Relative to the length-two results, the present scope includes arbitrary fixed length and admissible binary rings, with all permitted nonzero power values. Their possible valuations are $0,d,2d,\ldots<m$. Nonunit layers occur when $m>d$.

We resolve the power-value character sums by conductor, the least $r\in\{0,\ldots,m\}$ for which an additive character depends only on the residue class modulo $p^rR$. The resulting bounds lead to the finite counting inequality, from which we obtain the uniform error estimates by estimating the cycle coefficients. We then complete the classification using exact counts for $n=1$ and $n=2$. Appendix~\ref{app:fieldexact} recalls Li and Wan's exact field count and derives its relative-error normalisation and fixed-size expansion.

\section{Notation and conventions}\label{sec:notation}

Throughout the paper, $p$ is a prime. Let
\[
 R=\GR(p^m,p^{mf}),\qquad q=p^f,\qquad
 m,f\in\Z_{\geq1},\qquad d\in\Z_{\geq2},\qquad p\nmid d.
\]
Fix the residue-field identification $R/pR=\F_q$. For $1\leq s\leq m$, put
\[
 R_s=R/p^sR,\quad
 H_s=\{u^d:u\in R_s^\times\},\quad
 H=H_1=(\F_q^\times)^d,\quad g=\gcd(d,q-1).
\]
Write
\[
 P=\{x^d:x\in R\}\setminus\{0\},\qquad N=|P|,\qquad
 J=\left\lfloor\frac{m-1}{d}\right\rfloor.
\]
Write $\widehat R$ for the additive character group of $R$. For $\psi\in\widehat R$, define its conductor by
\[
 c(\psi)=\min\{r\in\{0,\ldots,m\}:\psi|_{p^rR}=1\}.
\]
The trivial additive character is denoted by $1$, and $c(\psi)=0$ precisely when $\psi=1$. Sums over empty index sets are zero. We write $\one_E$ for the indicator of a condition $E$ and $\vp(i)$ for the $p$-adic valuation of a positive integer $i$.

\section{Power values and exact character sums}

The Fourier count requires power-value character sums resolved by conductor. We first decompose the value set by valuation, then calculate the layer sums and the effect of integer multiplication.

The complete fibres of $H_s$ over $H$ determine the layer sums. For $0\leq t<m$, the map $\mu_t:R_{m-t}\to p^tR$ is given by multiplication by $p^t$ and sends units to elements of valuation exactly $t$. The following proposition also accounts for the changes in conductor and induced field character that can arise from integer multiplication.

\begin{proposition}[Power values and exact character sums]\label{prop:power-exact}
For $0\leq t<m$, the map
\begin{equation}\label{eq:mu}
 \mu_t:R_{m-t}\longrightarrow p^tR,\qquad
 \mu_t(x+p^{m-t}R)=p^tx
\end{equation}
is a well-defined isomorphism of \(R\)-modules. Interpret $p^tH_{m-t}$ as the image $\mu_t(H_{m-t})$. Then
\begin{equation}\label{eq:decomp}
 P=\bigsqcup_{j=0}^{J}\mu_{jd}(H_{m-jd}),\qquad
 N:=|P|=\frac{q-1}{g}\sum_{j=0}^{J}q^{m-jd-1}.
\end{equation}

If $c(\psi)=r>0$, define
\begin{equation}\label{eq:lambda}
  \lambda_\psi(h)=\psi(p^{r-1}\widetilde h)
  \qquad(h\in\F_q),
\end{equation}
where $\widetilde h\in R$ is any lift of $h$. This is independent of the lift and is a nontrivial additive character of $\F_q$.
Put
\[
  \eta_H(\lambda)=\sum_{h\in H}\lambda(h).
\]
The exact layer sums are
\[
  \sum_{y\in\mu_{jd}(H_{m-jd})}\psi(y)=
  \begin{cases}
    |H|q^{m-jd-1},&r\leq jd,\\
    q^{m-r}\eta_H(\lambda_\psi),&r=jd+1,\\
    0,&r\geq jd+2,
  \end{cases}
  \qquad (0\leq j\leq J,\ 1\leq r\leq m).
\]

To express the sum over all valuation layers, put
\[
  A_r=\frac{q-1}{g}
      \sum_{j=\lceil r/d\rceil}^{J}q^{m-jd-1}
      \qquad(1\leq r\leq m).
\]
Then, for $S(\psi)=\sum_{y\in P}\psi(y)$, we have
\begin{equation}\label{eq:S}
  S(\psi)=
  \begin{cases}
    N,&r=0,\\[2pt]
    A_r+\one_{d\mid(r-1)}q^{m-r}
    \eta_H(\lambda_\psi),&1\leq r\leq m.
  \end{cases}
\end{equation}

Let $\psi$ have conductor $c$, and let $i\geq1$ be an integer. Write $i=p^ea$, with $e=\vp(i)$ and $p\nmid a$, and let $\bar a\in\F_p^\times\subseteq\F_q^\times$ be its residue. For $\psi_i(x)=\psi(ix)$, one has
\begin{equation}\label{eq:ci}
 c_i:=c(\psi_i)=\max\{c-e,0\}.
\end{equation}
If $c_i>0$, then
\begin{equation}\label{eq:scale}
 \lambda_{\psi_i}(h)=\lambda_\psi(\bar a h)
 \quad(h\in\F_q).
\end{equation}
Consequently the complete twisted sum is
\begin{equation}\label{eq:Si}
 S_i(\psi):=\sum_{y\in P}\psi(iy)=
 \begin{cases}
 N,&c_i=0,\\[2pt]
 \displaystyle\frac{q-1}{g}
 \sum_{j=\lceil c_i/d\rceil}^{J}q^{m-jd-1}
 +\one_{d\mid(c_i-1)}q^{m-c_i}
 \sum_{h\in H}\lambda_\psi(\bar a h),&1\leq c_i\leq m.
 \end{cases}
\end{equation}
In the second branch $c>0$, so $\lambda_\psi$ is defined.

The number of additive characters of each conductor is
\begin{equation}\label{eq:counts}
 \#\{\psi\in\widehat R:c(\psi)=r\}=
 \begin{cases}
 1,&r=0,\\
 q^r-q^{r-1},&1\leq r\leq m.
 \end{cases}
\end{equation}
\end{proposition}

\begin{proof}
Since $R$ is a free module of rank $f$ over $\Z/p^m\Z$, we have
\begin{equation}\label{eq:annihilator}
 \Ann_R(p^t)=p^{m-t}R\quad(0\leq t\leq m).
\end{equation}
The annihilator identity makes \eqref{eq:mu} well-defined and injective. Surjectivity and $R$-linearity follow from multiplication by $p^t$. At $t=0$ this is the canonical identification $R_m=R/p^mR=R$.

We also have $|R_s|=q^s$. The kernel of $R_s^\times\to\F_q^\times$ is the finite abelian $p$-group $1+pR_s$, of order $q^{s-1}$. Since $p\nmid d$, the $d$th-power map is an automorphism of this kernel.

If $v\in R_s^\times$ has residue in $H$, choose $w\in R_s^\times$ whose residue is a $d$th root of the residue of $v$. Then $vw^{-d}\in1+pR_s$ is a $d$th power, so $v$ is also a $d$th power. The converse follows by reduction. Thus $H_s$ is the full inverse image of $H$. Since $\F_q^\times$ is cyclic, its $d$th-power map has kernel of size $g$, so
\[
 |H_s|=|H|q^{s-1}=\frac{q-1}{g}q^{s-1}.
\]

Every nonzero $x\in R$ can be written $x=p^ju$ with $u$ a unit. Its $d$th power is nonzero exactly when $jd<m$, and then has valuation $jd$. Hence $x^d$ belongs to $\mu_{jd}(H_{m-jd})$. Conversely, if $v\in R_{m-jd}^\times$, choose a unit lift $w\in R^\times$. Then
\[
 \mu_{jd}(v^d)=p^{jd}w^d=(p^jw)^d.
\]
The different images in \eqref{eq:decomp} have different valuations. This proves the disjoint decomposition and its cardinality.

Let $c(\psi)=r>0$. Two lifts differ by $pR$, so multiplication by $p^{r-1}$ puts their difference in $p^rR$, on which $\psi$ is trivial. Thus \eqref{eq:lambda} is lift-independent. It defines an additive character of $\F_q$ because $\psi$ is additive, and it is nontrivial because $\psi$ is nontrivial on $p^{r-1}R$.

For $0\leq j\leq J$, choose lifts $\widetilde h\in R$ of $h\in H$. Since $H_{m-jd}$ is the full inverse image of $H$, we have
\[
 \mu_{jd}(H_{m-jd})
 =\bigsqcup_{h\in H}
   \bigl(p^{jd}\widetilde h+p^{jd+1}R\bigr).
\]
Hence the character sum over this layer is
\[
 \sum_{h\in H}\psi(p^{jd}\widetilde h)
 \sum_{z\in p^{jd+1}R}\psi(z).
\]
If $r\leq jd$, the character is trivial on the entire layer, giving $|H|q^{m-jd-1}$. If $r=jd+1$, the inner sum is $q^{m-r}$ and $\psi(p^{jd}\widetilde h)=\lambda_\psi(h)$, giving $q^{m-r}\eta_H(\lambda_\psi)$. If $r\geq jd+2$, the restriction of $\psi$ to $p^{jd+1}R$ is nontrivial, since otherwise its conductor would be at most $jd+1$. The inner sum therefore vanishes by character orthogonality.

The layers with $jd\geq r$ contribute $A_r$ in total. A boundary layer exists precisely when $d\mid(r-1)$, and its index is $(r-1)/d$. For $1\leq r\leq m$ this index automatically lies between $0$ and $J$. Summing the three cases proves the positive-conductor branch of \eqref{eq:S}. For conductor zero the character is trivial on $R$, so $S(1)=N$ directly.

For the twist, the image of the integer $a$ is a unit. If $e\geq c$, then $iR\subseteq p^cR$, so $\psi_i=1$. If $e<c$, the character $\psi_i$ is trivial on $p^{c-e}R$ and is nontrivial on $p^{c-e-1}R$, because
\[
 i p^{c-e-1}R=p^{c-1}aR=p^{c-1}R.
\]
This proves \eqref{eq:ci}, including $c=0$. In the positive case, if $\widetilde h$ lifts $h$, then
\begin{align*}
 \lambda_{\psi_i}(h)
 &=\psi_i(p^{c_i-1}\widetilde h)\\
 &=\psi(p^{c-1}a\widetilde h)
 =\lambda_\psi(\bar a h),
\end{align*}
since $a\widetilde h$ lifts $\bar a h$. Formula \eqref{eq:Si} follows by applying \eqref{eq:S} to $\psi_i$.

Finally, for $1\leq r\leq m$, characters of conductor at most $r$ factor through $R_r$, whose additive group $(\Z/p^r\Z)^f$ has $q^r$ characters. Subtracting the $q^{r-1}$ characters of conductor at most $r-1$ proves \eqref{eq:counts}. Conductor zero consists of the unique trivial character.
\end{proof}

\section{Finite-field estimate and uniform conductor bounds}

Gauss sums bound the remaining finite-field subgroup sums. Substitution into the exact layer formula gives nonnegative conductor bounds for all integer-multiplication twists.

\subsection{The finite-field estimate}
Let $\varepsilon$ denote the trivial multiplicative character of $\F_q^\times$, and set
\[
 H^\perp=\{\chi\in\widehat{\F_q^\times}:\chi|_H=1\},\qquad
 G(\chi,\lambda)=\sum_{x\in\F_q^\times}\chi(x)\lambda(x).
\]
Here the hat denotes the multiplicative character group. We use the classical Gauss sum identities of \cite[Lemma~2.5(i), p.~153]{DZ} and \cite[Lemma~4.1(i)--(iv), p.~381]{Katz}. The latter applies to finite fields of arbitrary characteristic and extension degree.

\begin{lemma}[Subgroup detector and uniform bound]\label{lem:field}
For $x\in\F_q^\times$,
\begin{equation}\label{eq:detector}
 \one_H(x)=\frac1g\sum_{\chi\in H^\perp}\chi(x).
\end{equation}
For every nontrivial additive character $\lambda$ of $\F_q$,
\begin{equation}\label{eq:KH}
 |\eta_H(\lambda)|\leq K_H,
 \qquad K_H=\frac{1+(g-1)\sqrt q}{g}.
\end{equation}
\end{lemma}
\begin{proof}
Choose a generator $\alpha$ of $\F_q^\times$. Then $H=\langle\alpha^g\rangle$. If $\omega(\alpha)=\exp(2\pi\mathrm i/(q-1))$, then
\[
 H^\perp=\{\omega^{j(q-1)/g}:0\leq j<g\}.
\]
For $x=\alpha^k$, the right-hand side of \eqref{eq:detector} is
\[
 \frac1g\sum_{j=0}^{g-1}\exp(2\pi\mathrm i jk/g)
 =\one_{g\mid k}=\one_H(x).
\]
In particular, the detector contains exactly $g$ characters, including $\varepsilon$ once.

Orthogonality gives all three trivial-character cases:
\[
 G(\varepsilon,1)=q-1,\qquad
 G(\varepsilon,\lambda)=-1\quad(\lambda\ne1),\qquad
 G(\chi,1)=0\quad(\chi\ne\varepsilon).
\]
For nontrivial $\chi,\lambda$, character values have modulus one and the substitution $x=ty$ gives
\begin{align*}
 |G(\chi,\lambda)|^2
 &=\sum_{x,y\in\F_q^\times}
       \chi(x)\overline{\chi(y)}\lambda(x-y)\\
 &=\sum_{t\in\F_q^\times}\chi(t)
       \sum_{y\in\F_q^\times}\lambda((t-1)y)
 =(q-1)-\sum_{\substack{t\in\F_q^\times\\t\ne1}}\chi(t)=q.
\end{align*}
The inner sum is $q-1$ at $t=1$ and $-1$ otherwise, while $\sum_{t\ne1}\chi(t)=-1$. Inserting \eqref{eq:detector} now yields
\[
 \eta_H(\lambda)=\frac1g\left(-1+
 \sum_{\substack{\chi\in H^\perp\\\chi\ne\varepsilon}}
 G(\chi,\lambda)\right).
\]
The triangle inequality proves \eqref{eq:KH}, uniformly over all nontrivial additive characters. The contribution $-1$ remains when $g=1$, in which case $\eta_H(\lambda)=-1$ and $K_H=1$.
\end{proof}

\subsection{Uniform conductor bounds}
Define
\begin{equation}\label{eq:B}
 B_0=N,\qquad
 B_r=A_r+\one_{d\mid(r-1)}q^{m-r}K_H
 \quad(1\leq r\leq m).
\end{equation}
These numbers are nonnegative. By \eqref{eq:S} and Lemma~\ref{lem:field},
\[
 |S(\psi)|\leq B_{c(\psi)}
 \quad(\psi\in\widehat R).
\]
The boundary factor $q^{m-r}$ counts residue lifts. For $0\leq c\leq m$, put
\[
 r_i(c)=\max\{c-\vp(i),0\}\quad(i\geq1).
\]
For $c(\psi)=c$, applying the bound to $\psi_i$ gives
\begin{equation}\label{eq:twistbound}
 |S_i(\psi)|\leq B_{r_i(c)}.
\end{equation}
If $r_i(c)>0$, the induced character in \eqref{eq:scale} is nontrivial because multiplication by $\bar a$ is a bijection of $\F_q$. If $p^c\mid i$, the twist is trivial and $S_i(\psi)=N=B_0$. These terms remain in the majorant even when the original character $\psi$ is nontrivial.

\smallskip\noindent\textit{Field and unit-only specialisations.}
For $m=1$, one has $J=0$, $N=(q-1)/g$ and $B_1=K_H$. Appendix~\ref{app:fieldexact} treats $g=1$, and the general estimates and pair obstruction cover the remaining indices. If $d\geq m$, then $J=0$ and $P$ consists only of power units, with $N=(q-1)q^{m-1}/g$, $B_1=q^{m-1}K_H$ and $B_r=0$ for $2\leq r\leq m$. The indices $r_i(c)$ still permit lower-conductor or trivial twists of high-conductor characters.

\section{Subset counting and the finite error inequality}

Fourier inversion separates the average count as the trivial-character contribution. The twisted-sum bounds control the remaining coefficients.

Set
\[
 F_\psi(z)=\prod_{y\in P}(1+z\psi(y)).
\]
Write $[z^n]F(z)$ for the coefficient of $z^n$ in a formal power series. The empty subset has sum zero, and we use $\binom Nn=0$ for $n>N$. Expanding the product gives the fixed-cardinality identity
\begin{equation}\label{eq:subsetcoeff}
 [z^n]F_\psi(z)=
 \sum_{\substack{T\subseteq P\\|T|=n}}
 \psi\left(\sum_{y\in T}y\right).
\end{equation}

The cycle form of the distinct-coordinate sieve groups tuples by identifying coordinates. A cycle of length $i$ uses one value $y$ in $i$ positions, producing the character weight $\psi(iy)$ and hence the sum $S_i(\psi)$. For a character of conductor $c$, the nonnegative coefficient below bounds $|[z^n]F_\psi(z)|$ by replacing these twisted sums with $B_{r_i(c)}$. For $0\leq c\leq m$ and an integer $n\geq0$, define
\begin{equation}\label{eq:C}
 C_c(n)=[z^n]\exp\left(
             \sum_{i=1}^{n}\frac{B_{r_i(c)}}i z^i\right).
\end{equation}

For $1\leq n\leq N$, the signed coefficient identity below is the generating-function form of Li and Wan's finite-ring sieve \cite[Lemma~3.1 and its proof, pp.~8--9]{LW18}. We derive the coefficient majorant directly from $F_\psi(z)$ and \eqref{eq:twistbound}.

\begin{theorem}[Conductor-grouped error inequality]\label{thm:error}
For every $b\in R$ and integer $n\geq0$,
\begin{equation}\label{eq:error}
 \left|N_P(n,b)-q^{-m}\binom Nn\right|
 \leq q^{-m}\sum_{c=1}^{m}(q^c-q^{c-1})
 [z^n]\exp\left(
 \sum_{i=1}^{n}\frac{B_{\max\{c-\vp(i),0\}}}{i}z^i\right).
\end{equation}
\end{theorem}
\begin{proof}
The additive group of $R$ is $(\Z/p^m\Z)^f$, so
\[
 q^{-m}\sum_{\psi\in\widehat R}\psi(x)=\one_{x=0}
 \quad(x\in R):
\]
the coordinate sums are finite geometric series. Apply this to $x=\sum_{y\in T}y-b$ and use \eqref{eq:subsetcoeff}. Interchanging finite sums yields exact Fourier inversion:
\begin{equation}\label{eq:Fourier}
 N_P(n,b)=q^{-m}\sum_{\psi\in\widehat R}
            \psi(-b)[z^n]F_\psi(z).
\end{equation}
For the trivial character, $F_1(z)=(1+z)^N$, so its contribution is $q^{-m}\binom Nn$.

Work in $\C[[z]]$. Since $\psi(y)^i=\psi(iy)$, the formal logarithm gives
\[
 \log F_\psi(z)=\sum_{i\geq1}
                \frac{(-1)^{i-1}S_i(\psi)}i z^i.
\]
Truncating at degree $n$ gives
\begin{equation}\label{eq:truncate}
 F_\psi(z)\equiv\exp\left(
 \sum_{i=1}^{n}\frac{(-1)^{i-1}S_i(\psi)}i z^i\right)
 \pmod{z^{n+1}}.
\end{equation}
Put $c=c(\psi)$. If $n\geq1$, expansion of the exponential in \eqref{eq:truncate} expresses its degree-$n$ coefficient as
\[
 \sum_{\substack{k_1,\ldots,k_n\geq0\\\sum_{i=1}^{n}ik_i=n}}
 \prod_{i=1}^{n}
 \frac{\bigl((-1)^{i-1}S_i(\psi)/i\bigr)^{k_i}}{k_i!}.
\]
Taking termwise absolute values and applying \eqref{eq:twistbound} bounds this coefficient by $C_c(n)$ in modulus:
\begin{equation}\label{eq:majorant}
 |[z^n]F_\psi(z)|\leq C_c(n).
\end{equation}
For $n=0$, the exponent is an empty sum and both sides are one. All divisions by $i$ take place in $\C[[z]]$.

Subtract the trivial-character contribution from \eqref{eq:Fourier}. All character values have modulus one. The triangle inequality and \eqref{eq:majorant}, followed by grouping the nontrivial characters according to \eqref{eq:counts}, give
\[
 \left|N_P(n,b)-q^{-m}\binom Nn\right|
 \leq q^{-m}\sum_{\psi\ne1}C_{c(\psi)}(n)
 =q^{-m}\sum_{c=1}^{m}(q^c-q^{c-1})C_c(n).
\]
For $n>N$, the subset count and the product coefficient vanish, as does $\binom Nn$. The nonnegative majorant continues to give a valid inequality.
\end{proof}

\section{Relative normalisation and size comparisons}

To extract fixed-size asymptotics from the finite inequality, we compare the binomial denominator and the cycle contributions uniformly.

For $1\leq n\leq N$, define the uniform relative error and its nonnegative majorant by
\begin{align}
 \rho_n(q)&=\sup_{b\in R}
 \left|\frac{N_P(n,b)}{q^{-m}\binom Nn}-1\right|,\label{eq:rho}\\
 \mathcal M_n(q)&=
 \frac{\displaystyle\sum_{c=1}^{m}(q^c-q^{c-1})C_c(n)}{\binom Nn}.
 \label{eq:M}
\end{align}
Dividing \eqref{eq:error} by the positive main term gives $\rho_n(q)\leq\mathcal M_n(q)$. The majorant may overestimate the error by discarding cancellation between Fourier terms.

Throughout the asymptotic analysis, $m,d,n$ are fixed positive integers with $d\geq2$, and $q=p^f\to\infty$ through admissible prime powers. Both $p$ and $f$ may vary, as may $g=\gcd(d,q-1)$, with $1\leq g\leq d$.

\begin{lemma}\label{lem:size}
For every admissible $q$,
\[
 \frac{q^m}{2g}\leq N\leq\frac{q^m}{g},\qquad
 N=\frac{q^m}{g}\bigl(1+O(q^{-1})\bigr).
\]
The implied constant can be chosen absolute for $d\geq2$ and $m\geq1$. For $1\leq r\leq m$,
\[
 A_r\leq\frac{q^{m-r}}{g},\qquad
 B_r\leq2q^{m-r+1/2}.
\]
If $g=1$, the stronger comparison $B_r\leq2q^{m-r}$ holds. If $q^m\geq4gn$, then $N\geq2n$ and
\begin{equation}\label{eq:binlower}
 \binom Nn\geq\frac{q^{mn}}{(4g)^n n!}.
\end{equation}
\end{lemma}

\begin{proof}
Write $Q=q^m$. Formula \eqref{eq:decomp} becomes
\[
 N=\frac Qg(1-q^{-1})\sum_{j=0}^{J}q^{-jd}.
\]
The first summand gives the lower bound. The geometric comparison gives
\[
 (1-q^{-1})\sum_{j=0}^{J}q^{-jd}
 \leq\frac{1-q^{-1}}{1-q^{-d}}\leq1.
\]
The remaining terms after the first are $O(q^{-2})$, uniformly under $d\geq2$, proving the asymptotic comparison. If $a=\lceil r/d\rceil\leq J$, the same argument gives $A_r\leq q^{m-da}/g\leq q^{m-r}/g$. Otherwise $A_r=0$. Since $K_H\leq\sqrt q$, \eqref{eq:B} gives the first bound on $B_r$. For $g=1$, $K_H=1$, giving the stronger one. Finally $N\geq Q/(2g)$ and $Q\geq4gn$ imply $N\geq2n$. Each of the $n$ factors in the falling factorial is at least $N/2$, proving \eqref{eq:binlower}.
\end{proof}

Let
\[
 \mathcal P_n=\{(k_1,\ldots,k_n)\in\Z_{\geq0}^n:
                    \textstyle\sum_{i=1}^{n}ik_i=n\},\qquad
 \ell(k)=\sum_{i=1}^{n}k_i.
\]
Expanding the exponential in \eqref{eq:C} gives the finite identity
\begin{equation}\label{eq:cycles}
 C_c(n)=\sum_{k\in\mathcal P_n}
 \prod_{i=1}^{n}\frac{B_{r_i(c)}^{k_i}}{i^{k_i}k_i!}.
\end{equation}
A factor with $k_i=0$ is understood to be one, including when $B_{r_i(c)}=0$. Here $k_i$ records cycles of length $i$. A cycle with $p^c\mid i$ uses the full value $B_0=N$.

When $p>n$ and $1\leq c\leq m$, every $r_i(c)$ with $1\leq i\leq n$ equals $c$. In this regime,
\[
 C_c(n)=[z^n](1-z)^{-B_c}
 =\frac{B_c(B_c+1)\cdots(B_c+n-1)}{n!}.
\]
The identity follows from the formal differential equation for $(1-z)^{-B_c}$ and holds for every nonnegative real $B_c$, including $B_c=0$.

In the field case $m=1$, the positive-conductor coefficient $C_1(n)$ in \eqref{eq:cycles} is the cycle coefficient retained in Zhu and Wan's subgroup estimates \cite[proofs of Theorems~4.1 and~4.4]{ZW12}.

\section{The exact order of the nonnegative majorant}

We determine the order of $\mathcal M_n(q)$ by comparing the powers of $q$ in \eqref{eq:cycles}. We first identify the exponent $\beta_r$ for which $B_r\asymp q^{m-\beta_r}$ whenever $B_r>0$. Define $\beta_0=0$ and, for $1\leq r\leq m$, define
\begin{equation}\label{eq:beta}
 \beta_r=
 \begin{cases}
 r-\tfrac12,&g>1,\ d\mid(r-1),\\
 r,&g=1,\ d\mid(r-1),\\
 d\lceil r/d\rceil,&d\nmid(r-1),\ \lceil r/d\rceil\leq J,\\
 +\infty,&\text{otherwise}.
 \end{cases}
\end{equation}
Thus $\beta_r=+\infty$ exactly when $B_r=0$. For finite values, Lemma~\ref{lem:size} and the first surviving term give, uniformly in $p,f,g$,
\begin{equation}\label{eq:Bcompare}
 \frac1{2d}q^{m-\beta_r}\leq B_r\leq2q^{m-\beta_r}
 \qquad(0\leq r\leq m,\ \beta_r<\infty).
\end{equation}
For a boundary index, its possible full-layer term begins at valuation $r+d-1$, so it has strictly smaller order than its boundary term. For $g>1$, the leading boundary coefficient $(g-1)/g$ is at least $1/2$.

A pair $(c,k)\in\{1,\ldots,m\}\times\mathcal P_n$ is called supported when every $r_i(c)$ with $k_i>0$ has finite $\beta_{r_i(c)}$. Only supported patterns contribute to the majorant. For such pairs put
\begin{equation}\label{eq:gap}
 w_c(i)=m(i-1)+\beta_{r_i(c)},\qquad
 \Delta(c,k)=\sum_{\substack{1\leq i\leq n\\ k_i>0}}k_iw_c(i)-c.
\end{equation}
After including the conductor multiplicity and normalising by $\binom Nn$, the exponent is
\[
 c+\sum_{\substack{1\leq i\leq n\\k_i>0}}
 k_i(m-\beta_{r_i(c)})-mn=-\Delta(c,k).
\]
Every sum involving the weights of a supported partition is taken only over indices with $k_i>0$.

\begin{proposition}\label{prop:exponent}
Fix $m\geq1$, $d\geq2$ and $n\geq1$. The minimum supported gap is
\begin{equation}\label{eq:gamma}
 \gamma=
 \begin{cases}
 n/2-1,&g>1,\\
 n-1,&g=1,\ m\geq2,\\
 n-1-\lfloor n/p\rfloor,&g=1,\ m=1.
 \end{cases}
\end{equation}
Moreover,
\begin{equation}\label{eq:asympmajorant}
 \mathcal M_n(q)\asymp_{m,d,n}q^{-\gamma}.
\end{equation}
The comparison constants are uniform in all admissible $p,f,g$ for sufficiently large $q$. The exponent is evaluated pointwise when the residue characteristic or subgroup index varies. The comparison concerns $\mathcal M_n(q)$. Its implication for $\rho_n(q)$ is the upper bound $\rho_n(q)\leq\mathcal M_n(q)$.
\end{proposition}

\begin{proof}
First suppose $g>1$. At every positive index with finite $\beta_r$, we have $\beta_r\geq r-1/2$, and $\beta_0=0$. For a used length $i=1$,
\[
 w_c(1)=\beta_c\geq c-1/2=c+1/2-1.
\]
For a used length $i\geq2$, all finite $\beta_r$ are nonnegative, so
\[
 w_c(i)\geq m(i-1)\geq c(i-1)\geq c+i/2-1,
\]
where the last difference is $(i-2)(c-1/2)\geq0$. Summing over a supported partition gives
\begin{equation}\label{eq:generalgap}
 \Delta(c,k)\geq n/2-1+(\ell(k)-1)(c-1)\geq n/2-1.
\end{equation}
Conductor $c=1$ and $k_1=n$ attain equality because $\beta_1=1/2$.

Suppose next $g=1$ and $m\geq2$. Then $\beta_r\geq r$ for all finite indices, including $r=0$. Write $e=\vp(i)$. The elementary inequality $e\leq i-1$ follows from $2^e\leq i$ and $2^e\geq e+1$. If $r_i(c)>0$, then
\[
 w_c(i)\geq m(i-1)+c-e\geq c+i-1,
\]
since $(m-1)(i-1)\geq e$. If $r_i(c)=0$, then $c\leq e\leq i-1$, and $(m-1)(i-1)\geq c$, which gives the same inequality. Thus
\[
 \Delta(c,k)\geq n-1+(\ell(k)-1)(c-1)\geq n-1.
\]
Conductor one with all singleton cycles attains equality, since $\beta_1=1$.

Finally, if $g=1,m=1$, only $c=1$ occurs and
\[
 w_1(i)=\begin{cases}i,&p\nmid i,\\i-1,&p\mid i.\end{cases}
\]
Therefore
\[
 \Delta(1,k)=n-1-\sum_{p\mid i}k_i.
\]
There are at most $\lfloor n/p\rfloor$ cycles with lengths divisible by $p$. This number is attained by that many $p$-cycles and the remaining singleton cycles. This proves \eqref{eq:gamma}.

By \eqref{eq:Bcompare}, $(q^c-q^{c-1})\asymp q^c$, Lemma~\ref{lem:size} and \eqref{eq:cycles}, each normalised supported summand is comparable to $q^{-\Delta(c,k)}$, with constants depending only on $m,d,n$. There are finitely many terms and all are nonnegative. The attaining terms above therefore give \eqref{eq:asympmajorant}. Uniformity in $p,f,g$ follows from these estimates and the finitely many valuation patterns of $1,\ldots,n$.
\end{proof}

\begin{remark}[Dominant contributions]\label{rem:dominant}
When $g>1$ and $n\geq3$, the conductor-one singleton contribution is
\[
 \frac{(q-1)B_1^n}{n!\binom Nn}
 =(g-1)^n q^{1-n/2}\bigl(1+O_{m,d,n}(q^{-1/2})\bigr).
\]
This estimate is uniform over $2\leq g\leq d$. All supported terms with conductor $c>1$ have a strictly larger gap: in \eqref{eq:generalgap}, either $\ell>1$ gives strictness, or a single cycle of length $n\geq3$ gives strictness in the last inequality for $i\geq2$. For $c=1$, every nonsingleton cycle has strict slack unless $m=1,p=2,i=2$. In that exceptional regime every partition into cycles of lengths one and two attains the same gap $n/2-1$.

For $g=1,m=1$, the dominant power is obtained by maximising the number of $p$-divisible cycles.
\end{remark}

\section{Uniform relative-error theorem and positivity}

The preceding gap calculation determines the order of the majorant in $q$. Summing the cycle weights now gives an explicit relative-error bound and sufficient conditions for positivity at every target.

\begin{theorem}[Uniform fixed-size estimates]\label{thm:rates}
Let $m,d,n$ be fixed, with $m\geq1$, $d\geq2$ and $n\geq1$. For an admissible $q=p^f$ satisfying $q^m\geq4gn$, put
\[
 D_{m,g,n}=m(n+1)n!(4g)^n.
\]
Then $1\leq n\leq N$ and
\begin{equation}\label{eq:explicit}
 \rho_n(q)\leq\mathcal M_n(q)\leq D_{m,g,n}q^{-\gamma},
\end{equation}
where $\gamma$ is \eqref{eq:gamma}. In particular, for every fixed $n\geq3$ and every admissible sequence $q=p^f\to\infty$,
\begin{equation}\label{eq:universal}
 N_P(n,b)=q^{-m}\binom Nn
 \left(1+O_{m,d,n}(q^{1-n/2})\right)
 \quad\text{uniformly in }b\in R.
\end{equation}
An explicit constant in \eqref{eq:universal} is $D_{m,d,n}=m(n+1)n!(4d)^n$, valid once $q^m\geq4dn$.

When $g=1,m\geq2$, the stronger rate is $O_{m,n}(q^{1-n})$, and when $g=1,m=1$, it is $O_n(q^{1-n+\lfloor n/p\rfloor})$, uniformly in the residue characteristic.

Whenever the hypotheses of \eqref{eq:explicit} hold and $D_{m,g,n}q^{-\gamma}<1$, every count $N_P(n,b)$ is positive. For the universal $n\geq3$ statement, sufficient explicit conditions are
\[
 q^m\geq4dn,\qquad
 q^{n/2-1}>D_{m,d,n}.
\]
\end{theorem}

\begin{proof}
Expand $C_c(n)$ in $\mathcal M_n(q)$ using \eqref{eq:cycles}. For each supported term, bound the conductor multiplicity by $q^c$ and apply \eqref{eq:Bcompare} to obtain the numerator bound
\[
 \frac{2^{\ell(k)}}{\prod_i i^{k_i}k_i!}
 q^{c+\sum_{\substack{1\leq i\leq n\\ k_i>0}}k_i(m-\beta_{r_i(c)})}.
\]
Dividing by \eqref{eq:binlower} gives a contribution of at most
\[
 (4g)^n n!
 \frac{2^{\ell(k)}}{\prod_i i^{k_i}k_i!}
 q^{-\Delta(c,k)}.
\]
By Proposition~\ref{prop:exponent}, every gap is at least $\gamma$. The partition weights have the exact sum
\[
 \sum_{k\in\mathcal P_n}
 \frac{2^{\ell(k)}}{\prod_i i^{k_i}k_i!}
 =[z^n]\exp\left(2\sum_{i\geq1}\frac{z^i}{i}\right)
 =[z^n](1-z)^{-2}=n+1.
\]
Summing over at most $m$ conductors proves \eqref{eq:explicit}. For $g=1,m=1$, $p\geq2$ implies
\[
 n-1-\lfloor n/p\rfloor\geq\lceil n/2\rceil-1\geq n/2-1.
\]
The other branches of \eqref{eq:gamma} also have gap at least $n/2-1$. Since $g\leq d$, the asserted universal rate follows. The decay exponent is positive for $n\geq3$, so the relative error tends to zero. The main term is positive in the stated domain. If its relative-error upper bound is strictly less than one, the count is strictly positive.
\end{proof}

\section{Small-size obstructions and the fixed-size classification}

The upper bounds supply sufficiency for equidistribution. Exact counts at target zero give the small-size obstructions needed for necessity.

\begin{proposition}[Exact small-size obstructions]\label{prop:obstructions}
For $n=1$, uniform relative equidistribution fails along every admissible sequence. For $n=2$, put $T=|\{y\in P:2y=0\}|$. Then
\begin{equation}\label{eq:pairs}
 N_P(2,0)=
 \begin{cases}
 0,&-1\notin H,\\
 (N-T)/2,&-1\in H,
 \end{cases}
\end{equation}
where
\begin{equation}\label{eq:T}
 T=
 \begin{cases}
 0,&p\ne2,\\
 N,&p=2,\ m=1,\\
 \one_{d\mid(m-1)}(q-1)/g,&p=2,\ m\geq2.
 \end{cases}
\end{equation}
If $-1\in H$ and $(p,m)\ne(2,1)$, then, for fixed $m,d$ as $q\to\infty$,
\begin{equation}\label{eq:pairratio}
 \frac{N_P(2,0)}{q^{-m}\binom N2}
 =g+O_{m,d}(q^{-1}).
\end{equation}
The estimate is uniform in $p,f,g$.
\end{proposition}

\begin{proof}
The value zero is excluded from $P$, so $N_P(1,0)=0$ whereas $q^{-m}N>0$. Its relative error equals one.

The full residue-fibre description in Proposition~\ref{prop:power-exact} gives $P=-P$ if $-1\in H$ and $P\cap(-P)=\varnothing$ otherwise, since the cosets $H$ and $-H$ are equal or disjoint, respectively. A distinct ordered pair with sum zero is $(y,-y)$ with $y\in P$ and $2y\ne0$. Division by two proves \eqref{eq:pairs}.

For odd $p$, multiplication by two is invertible. For $p=2,m=1$, it annihilates the whole ring. For $p=2,m\geq2$, \eqref{eq:annihilator} with $t=1$ gives $\Ann_R(2)=2^{m-1}R$. Its nonzero elements have valuation $m-1$. Therefore a layer of $P$ contributes only when $jd=m-1$, and that layer has size $(q-1)/g$. This proves \eqref{eq:T}. In this last case,
\[
 T/N=O_{m,d}(q^{1-m})=O_{m,d}(q^{-1}).
\]
In the cases relevant to \eqref{eq:pairratio},
\[
 \frac{N_P(2,0)}{q^{-m}\binom N2}
 =\frac{q^m(N-T)}{N(N-1)}
 =\frac{q^m}{N}\left(1-\frac TN\right)\frac N{N-1}.
\]
Lemma~\ref{lem:size} and $g\leq d$ give \eqref{eq:pairratio}.
\end{proof}

\begin{theorem}[Fixed-positive-size classification]\label{thm:classification}
Fix $m\geq1$, $d\geq2$, and a positive integer $n$. Along any admissible sequence $q_j=p_j^{f_j}\to\infty$, put
\[
 R_j=\GR(p_j^m,p_j^{mf_j}),\qquad
 P_j=\{x^d:x\in R_j\}\setminus\{0\},\qquad
 N_j=|P_j|,\qquad g_j=\gcd(d,q_j-1).
\]
Uniform relative equidistribution
\[
 \sup_{b\in R_j}
 \left|\frac{N_{P_j}(n,b)}{q_j^{-m}\binom{N_j}{n}}-1\right|\longrightarrow0
\]
has the following classification:
\begin{enumerate}
\item If $n=1$, it always fails.
\item If $n\geq3$, it always holds, with rate \eqref{eq:universal}.
\item If $n=2$, it holds if and only if, eventually,
\[
 g_j=1,\qquad \bigl(m\geq2\ \text{or}\ p_j\ne2\bigr).
\]
In the successful $n=2$ cases the relative error is $O_{m,d}(q_j^{-1})$.
\end{enumerate}
The quotient is used only when $1\leq n\leq N_j$, which holds for all sufficiently large $j$.
\end{theorem}

\begin{proof}
The first two assertions follow from Proposition~\ref{prop:obstructions} and Theorem~\ref{thm:rates}. If $n=2$ and eventually $g=1$, the exponent in \eqref{eq:gamma} is one for $m\geq2$, and also for $m=1,p\ne2$. The constants in \eqref{eq:explicit} are uniform, proving sufficiency.

If $g\geq2$ infinitely often, consider those terms of the sequence. At target zero, \eqref{eq:pairs} and \eqref{eq:pairratio} give a ratio of either zero or $g+O_{m,d}(q^{-1})$, except that field characteristic two always gives zero. Because $2\leq g\leq d$, the ratio's distance from one is bounded away from zero for sufficiently large $q$ on that subsequence. If $m=1,p=2$ infinitely often, the ratio at target zero is exactly zero. Either circumstance prevents uniform convergence and proves necessity.
\end{proof}

\begin{remark}[All extension degrees]
For a fixed admissible $p$ and $d\geq2$, equidistribution at $n=2$ cannot hold along \emph{all} extension degrees $f\to\infty$. Choose a prime $\ell\mid d$. Since $p\nmid d$, the residue of $p$ has a finite multiplicative order modulo $\ell$. For every positive multiple of that order, $\ell\mid(p^f-1)$, hence $g\geq\ell$ infinitely often. Theorem~\ref{thm:classification} applies. Selected extension-degree subsequences with $g=1$ can have the successful $n=2$ behaviour, subject to the field characteristic-two exception.
\end{remark}

In residue characteristic two, admissibility requires odd $d$. At $n=2$, the field case always fails, while length $m\geq2$ succeeds on eventual index-one sequences. The self-pair correction is \eqref{eq:T}. For even $d$, every admissible $p$ is odd and $g\geq2$ because $q-1$ is even, so pairs fail along every admissible sequence.

\appendix
\numberwithin{equation}{section}
\renewcommand{\theHequation}{appendix.\Alph{section}.\arabic{equation}}

\section{Technical comparison with the length-two literature}\label{app:length-two-comparison}

Throughout this comparison, $m=2$, $p$ is odd, $d\geq2$ and $p\nmid d$. Both papers of Ding and Zhou cover subset sums over the same distinct nonzero power-value set $P$, with $N=q(q-1)/g$.
Their 2019 exponent $k$, extension parameter $r$ and subgroup index correspond to our $d$, $f$ and $g$, respectively. The complete residue fibres have $q$ lifts and $g$ roots per value \cite[Lemma~2.9(1), pp.~155--156]{DZ}. Together with the character definitions and \cite[Lemma~2.8, p.~154]{DZ}, these give the exact sums $N$, $q\eta_H(\lambda_\psi)$ and $0$ at conductors $0$, $1$ and $2$. Integer twists lower conductors as in \eqref{eq:ci}, rescale the induced field character by the unit part when the lowered conductor is positive, and contribute the full value $N$ when trivial.

\subsection*{Consequences of the 2019 estimates}
The printed all-target bound has main term $q^{-2}\binom Nn$ \cite[Theorem~6.1, pp.~166--167]{DZ}. The proof estimates \cite[proof of Theorem~3.2, pp.~158--159]{DZ} and distinct-value conversion \cite[Section~6, pp.~166--167]{DZ} also yield, for fixed $d,n$ with $n\geq1$, relative-error bounds $O_{d,n}(q^{1-n/2})$ for $g>1$ and $O_n(q^{1-n})$ for $g=1$. These bounds use positive normalisation on $1\leq n\leq N$ and hold uniformly over unit targets along arbitrary admissible sequences, and over all targets along sequences with eventually $p\nmid n$.

\subsection*{Consequences of the 2023 exact formulas}
The sequel's constant-fibre formula counts distinct image values \cite[Theorem~5.11, p.~130]{DZII}. For $A=R^\times$ and $F(X)=X^d$ \cite[Corollary~5.12, p.~133]{DZII}, the image is $P$, each fibre has size $g$, and $|R\setminus A|=q$. Image distinctness is imposed after cyclewise fibre conversion, since distinct bases can have repeated images. Combining the exact image Fourier/sieve identities \cite[equations~(5.1)--(5.5), pp.~131--132]{DZII} with the character sums above gives our length-two coefficient majorant, with $q-1$ and $q^2-q$ characters of conductors one and two:
\[
 \left|N_P(n,b)-q^{-2}\binom Nn\right|
 \leq q^{-2}\bigl\{(q-1)C_1(n)+(q^2-q)C_2(n)\bigr\},
\]
where $C_c(n)$ is defined in \eqref{eq:C}. For fixed $d,n$ with $n\geq1$, division by the positive main term on $1\leq n\leq N$ gives all-target relative-error bounds $O_{d,n}(q^{1-n/2})$ for $g>1$ and $O_n(q^{1-n})$ for $g=1$, uniformly over admissible varying odd $p,f,g$, including $p\mid n$. In particular, index-one pairs have relative-error bound $O(q^{-1})$ and satisfy uniform relative equidistribution, already covered by the qualified 2019 derivation above. These rates follow from the exact formulas and are not stated separately in the sequel. The coefficient majorant bounds the actual error and need not have the same order.

\section{An exact formula in the index-one field case}\label{app:fieldexact}

We recall Li and Wan's exact subset-count formula \cite[Theorem~1.2]{LW08} and derive its relative error and fixed-size asymptotic expansion.

\begin{proposition}\label{prop:fieldexact}
Suppose $m=1$ and $g=1$, so $P=\F_q^\times$. For every $1\leq n\leq q-1$, put
\[
 L=\lfloor n/p\rfloor,\qquad
 a_n=(-1)^{n+L}\binom{q/p-1}{L}.
\]
Then the exact count is
\begin{equation}\label{eq:fieldexact}
 N_P(n,b)=\frac1q\left[
 \binom{q-1}{n}+(q\one_{b=0}-1)a_n\right].
\end{equation}
In particular,
\begin{equation}\label{eq:fieldrho}
 \rho_n(q)=
 \frac{(q-1)\binom{q/p-1}{L}}{\binom{q-1}{n}},
\end{equation}
attained at $b=0$. For fixed $d,n$, as $q\to\infty$ through admissible prime powers with $g=1$,
\begin{equation}\label{eq:fieldsharp}
 \rho_n(q)=\frac{n!}{p^L L!}\,
 q^{1+L-n}\bigl(1+O_n(q^{-1})\bigr).
\end{equation}
The error term is uniform in the prime characteristic. The parameter $L$, the exponent and the leading coefficient are evaluated at each $q$.
\end{proposition}

\begin{proof}
A nontrivial additive character of $\F_q$ has the $p$th roots of unity as its image, each with $q/p$ preimages. Using $\prod_{\zeta^p=1}(1+z\zeta)=1-(-z)^p$ and removing the factor at zero gives
\[
 F_\psi(z)=\frac{(1-(-z)^p)^{q/p}}{1+z}.
\]
The quotient is exact, and its degree-$n$ coefficient is
\[
 (-1)^n\sum_{j=0}^{L}(-1)^j\binom{q/p}{j}
 =(-1)^{n+L}\binom{q/p-1}{L}.
\]
Pascal's identity telescopes the sum, and $L\leq q/p-1$ follows from $n\leq q-1$. The coefficient is common to all nontrivial $\psi$. Orthogonality gives
\[
 \sum_{\psi\ne1}\psi(-b)=q\one_{b=0}-1,
\]
so \eqref{eq:Fourier} proves \eqref{eq:fieldexact}. The largest absolute multiplier is $q-1$, attained at $b=0$, giving \eqref{eq:fieldrho}.

For fixed $n$, if $p>n$ then $L=0$ and the numerator binomial is one. If $p\leq n$, only finitely many primes occur, and for each of them the falling-factorial expansions give
\[
 \binom{q/p-1}{L}=
 \frac{(q/p)^L}{L!}\bigl(1+O_n(q^{-1})\bigr),\qquad
 \binom{q-1}{n}=\frac{q^n}{n!}\bigl(1+O_n(q^{-1})\bigr).
\]
The same denominator expansion applies when $p>n$. Choosing the constants uniformly over the finitely many primes $p\leq n$ and substituting into \eqref{eq:fieldrho} proves \eqref{eq:fieldsharp} uniformly in the characteristic.
\end{proof}

\section*{Declaration of generative AI and AI-assisted technologies in the manuscript
preparation process}

OpenAI ChatGPT and Codex were used for literature and bibliographic assistance, manuscript editing and organisation.

The author independently verified the mathematical arguments and takes full responsibility for the content of the manuscript.
\bibliographystyle{plain}
\bibliography{ref}
\end{document}